\documentclass[11pt, reqno]{amsart}

\usepackage[margin=1in]{geometry}
\usepackage{amsmath,amssymb,amsthm,mathtools}
\usepackage{microtype}
\usepackage{xcolor}
\usepackage[colorlinks=true,linkcolor=blue!45!black,
citecolor=blue!45!black,urlcolor=blue!45!black]{hyperref}
\newcommand{\E}{\mathbb E}
\newcommand{\Ber}{\operatorname{Ber}}
\newcommand{\TV}{\mathrm{TV}}
\newcommand{\dd}{\mathrm d}

\newtheorem{theorem}{Theorem}[section]
\newtheorem{lemma}[theorem]{Lemma}
\newtheorem{corollary}[theorem]{Corollary}
\theoremstyle{remark}
\newtheorem{remark}[theorem]{Remark}
\numberwithin{equation}{section}

\title{Total Variation Distance between Product Distributions: an Analytic Proxy}
\author{Ariel Avital}
\author{Aryeh Kontorovich}
    \address{A.A and A.K.: Ben-Gurion University of the Negev}
    \email[Ariel Avital]{avitalq@post.bgu.ac.il}
    \email[Aryeh Kontorovich]{karyeh@cs.bgu.ac.il}
\author{Roman Vershynin}
    \address{R.V. and G.Z: University of California, Irvine}
    \email[Roman Vershynin]{rvershyn@uci.edu}
    \email[Guangyi Zou]{zouguangyi2001@gmail.com}
    \thanks{
    A.K. was supported by
    the Israel Science Foundation
(ISF 581/25)
and
the Binational Science Foundation
(BSF 2024243).
R.V. was supported by NSF Grant DMS 2451011 and U.S. Air Force Grant FA9550-25-1-0294.}
\author{Guangyi Zou}
\date{}

\begin{document}

\begin{abstract}
We characterize, up to universal constants, the total variation distance between finite products of arbitrary probability measures by a simple formula. A theorem of Lata\l a reduces this expression to one scalar equation.
As an application, we characterize the sample complexity of equal-prior binary hypothesis testing uniformly in the weak-detection regime, where squared Hellinger distance alone does not determine the answer,
and also
characterize the TV distance between multinomial distributions.
\end{abstract}
\maketitle

\section{Main results}

Total variation measures the accuracy of distributional approximations~\cite{LeCam1960} and determines the optimal error in equal-prior binary hypothesis testing~\cite{LeCamYang2000}. It does not tensorize exactly~\cite{KontorovichTensorization2025}; analytic bounds for Bernoulli products and sums appear in~\cite{KontorovichTVHomogenization2026,KontorovichHomogenizationPrinciple2026,AvitalKontorovichSalafatinos2026}.

Exact evaluation is $\#\mathsf{P}$-hard even for Bernoulli products~\cite{BhattacharyyaEtAl2023}, but randomized and deterministic polynomial-time relative-approximation schemes are available for finite-alphabet products~\cite{FengGuoJerrumWang2023,FengLiuLiu2024}. For fixed accuracy and confidence, the randomized algorithm of Anand, Benford, and Guo~\cite{AnandBenfordGuo2026} runs in time linear in the marginal-table size in the unit-cost arithmetic model.

Our goal is analytic: we give a dimension-free constant-factor formula for products of arbitrary probability measures. It reduces to one scalar equation.

Recall that, for probability measures \(P,Q\) on a measurable space \((\mathsf X,\mathcal F)\), the total variation is defined as
\[
{\TV}(P,Q):=\sup_{B\in\mathcal F}|P(B)-Q(B)|.
\]
Fix \(n\ge1\). For probability measures \(P_i,Q_i\) on measurable spaces \((\mathsf X_i,\mathcal F_i)\), our goal is to compute the TV distance between the  product measures
\[
\mathbf P:=\bigotimes_{i=1}^nP_i,\qquad
\mathbf Q:=\bigotimes_{i=1}^nQ_i.
\]
To do this, consider the midpoint measure defined by
\[
M_i:=\frac{P_i+Q_i}{2},\qquad
\mathbf M:=\bigotimes_{i=1}^nM_i.
\]
Define the \emph{midpoint score} in coordinate \(i\) as the
Radon--Nikodym derivative\footnote{If \(P_i,Q_i\) have densities
or probability mass functions \(p_i,q_i\), then
\(u_i(x)=(p_i(x)-q_i(x))/(p_i(x)+q_i(x))\), with \(0/0:=0\).
The Radon--Nikodym notation extends this definition to arbitrary
probability measures.}
\[
u_i:=\frac{\dd(P_i-Q_i)}{\dd(P_i+Q_i)}.
\]
Let \(X_1,\ldots,X_n\) be independent random variables with distribution \(X_i\sim M_i\), and set
\begin{equation}\label{eq: U}
U_i=u_i(X_i).
\end{equation}

\begin{theorem}
\label{thm:score}
We have
\begin{equation}
{\TV}(\mathbf P,\mathbf Q)
\asymp
\min\Bigl\{1,\E\Bigl(\sum_{i=1}^nU_i^2\Bigr)^{1/2}\Bigr\},
\label{eq:score-main}
\end{equation}
where the equivalence is up to a universal constant factor.\footnote{We write \(a\lesssim b\) if \(a\le Cb\) for a universal constant \(C\), write \(a\gtrsim b\) if \(b\lesssim a\), and write \(a\asymp b\) if both bounds hold.}
\end{theorem}

This result can be further simplified.
The expectation in \eqref{eq:score-main} can be characterized, up to universal constants, by a one-dimensional equation:
\begin{theorem}
\label{thm:proxy}
Assume that \(\mathbf P\ne \mathbf Q\). Let \(T>0\) be the unique solution of
\[
G(T)=\frac12,
\qquad\text{where}\qquad
G(t):=\sum_{i=1}^n\log\E\sqrt{1+\frac{U_i^2}{t^2}}.
\]
Otherwise, set \(T=0\). Then
\begin{equation}
{\TV}(\mathbf P,\mathbf Q)\asymp\min\{1,T\}.
\label{eq:proxy-TV}
\end{equation}
\end{theorem}

\section{The proofs}

\begin{proof}[Proof of Theorem~\ref{thm:score}]
By definition, we have \(M_i\)-almost surely:
\[
\frac{\dd P_i}{\dd M_i}=1+u_i,\qquad
\frac{\dd Q_i}{\dd M_i}=1-u_i.
\]
This allows us to express the TV distance between the marginals in terms of $U_i$:
\begin{equation} \label{eq: TV as U}
\TV(P_i,Q_i)=\frac{1}{2}\int \left|\frac{\dd P_i}{\dd M_i}-\frac{\dd Q_i}{\dd M_i}\right|\dd M_i=\int |u_i|\dd M_i=\E|U_i|.
\end{equation}
Similarly, we have
\begin{equation} \label{eq: Ui mean}
\E U_i=0, \qquad |U_i|\le 1 \quad \text{\((M_i\)-almost surely.)}
\end{equation}

Next, let's express the TV distance between the product distributions in terms of $U_i$.
For any \(x=(x_1,\ldots,x_n)\), we have, \(\mathbf M\)-almost surely:
\begin{equation}\label{equ:PQ-def}
\frac{\dd\mathbf P}{\dd\mathbf M}(x)
=\prod_{i=1}^n\bigl(1+u_i(x_i)\bigr),\qquad
\frac{\dd\mathbf Q}{\dd\mathbf M}(x)
=\prod_{i=1}^n\bigl(1-u_i(x_i)\bigr).
\end{equation}
Expanding the products, we see that the difference of these two product ratios is
their odd part:
\begin{align}
{\TV}(\mathbf P,\mathbf Q)
&=\frac12\int\Bigl|
\frac{\dd\mathbf P}{\dd\mathbf M}
-\frac{\dd\mathbf Q}{\dd\mathbf M}\Bigr|\,\dd\mathbf M
=\E|\Phi(U)|,
\label{eq:tv-phi}
\end{align}
where
\begin{equation}
\Phi(U):=
\frac12\Bigl\{\prod_i(1+U_i)-\prod_i(1-U_i)\Bigr\}
=\sum_{\substack{I\subseteq[n]\\|I|\ \mathrm{odd}}}
\prod_{i\in I}U_i.
\label{eq:odd-expansion}
\end{equation}
Let us linearize \(\Phi\) by writing
\[
\Phi(U)=S+R,
\qquad
S:=\sum_{i=1}^nU_i.
\]
Marcinkiewicz--Zygmund inequality with exponent \(1\) (see, e.g., \cite[Chapter~3, Section~8]{Gut2013}) allows us to  evaluate the main, linear term:
\begin{equation}
\label{eq:S=A}
\E|S|\asymp A, \qquad\text{where }A:=\E\Bigl(\sum_{i=1}^nU_i^2\Bigr)^{1/2}.
\end{equation}

Next, we want to bound the remainder term \(R\). It contains the odd terms of degrees at least three. Distinct square-free monomials are orthogonal, since $\E U_i = 0$ by \eqref{eq: Ui mean}. Hence
\begin{equation}
\E R^2
=\sum_{\substack{k\ge3\\k\ \mathrm{odd}}}
\sum_{|I|=k}\prod_{i\in I}\E U_i^2
\le
\sum_{\substack{k\ge3\\k\ \mathrm{odd}}}\frac{\rho^k}{k!},
\qquad
\text{where }
\rho:=\sum_{i=1}^n\E U_i^2.
\label{eq:R-second}
\end{equation}
In particular, if \(0\le\rho\le1\), we have $\E R^2\lesssim \rho^{3}$, so we obtain a bound on the remainder:
\begin{equation}
\E|R|\lesssim \rho^{3/2},
\quad \text{if }0\le\rho\le1.
\label{eq:R-first}
\end{equation}
We want to say this this is negligible compared to the main term, whose expected magnitude is about $A$ as we saw in \eqref{eq:S=A}. So let's bound $A$ below:
\begin{align}
A
&\ge1-\E e^{-\sum_iU_i^2}
&\text{(by \(\sqrt z\ge1-e^{-z}\))}\notag\\
&=1-\prod_i\E e^{-U_i^2}
&\text{(by independence)}\notag\\
&\ge 1-\prod_i\E (1-\frac12 U_i^2)
&\text{(since \(0\le U_i^2\le1\) by \eqref{eq: Ui mean})}\notag
\\
&=1-\prod_i (1-\frac12\E U_i^2)
\ge1-\prod_i e^{-\frac12\E U_i^2}
&\text{(by \(1-x\le e^{-x}\))}\notag
\\
&=1-e^{-\rho/2}
\gtrsim \min\{1,\rho\}.&
\label{eq:A-rho}
\end{align}
Combining this with \eqref{eq:R-first}, we obtain
\[
\E|R|\lesssim\sqrt\rho\,\E|S|,
\qquad
\text{if }0\le\rho\le1.
\]

\emph{Small total variance.}
Suppose that \(\rho\le\rho_0\), where \(\rho_0\in(0,1)\) is a
sufficiently small universal constant. Then
\(\E|R|\le\frac12\E|S|\), and hence, by
\eqref{eq:tv-phi} and \eqref{eq:S=A},\[
{\TV}(\mathbf P,\mathbf Q)
=\E|\Phi(U)|=\E|S+R|
\asymp\E|S|
\asymp A.
\]
Moreover, Jensen's inequality implies that $A \le \sqrt{\rho}<1$. This proves \eqref{eq:score-main} in this regime.

\emph{Large total variance.} If \(\rho>\rho_0\), we no longer linearize \(\Phi\). Using \(\min\{a,b\}\le\sqrt{ab}\) and the product structure, we obtain
\begin{align*}
1-{\TV}(\mathbf P,\mathbf Q)
&=\int\min\Bigl\{
\frac{\dd\mathbf P}{\dd\mathbf M},
\frac{\dd\mathbf Q}{\dd\mathbf M}
\Bigr\}\,\dd\mathbf M \\
&\le
\int\sqrt{
\frac{\dd\mathbf P}{\dd\mathbf M}
\frac{\dd\mathbf Q}{\dd\mathbf M}
}\,\dd\mathbf M
=\prod_i\E\sqrt{1-U_i^2}\tag{by \eqref{equ:PQ-def} and definition of $U_i$}
\\
&\le
\prod_i\Bigl(1-\frac12\E U_i^2\Bigr)
\le \prod_i e^{-\frac12\E U_i^2}=e^{-\rho/2}.
\end{align*}
Thus
\begin{equation}
{\TV}(\mathbf P,\mathbf Q)\ge 1-e^{-\rho/2}.
\label{eq:tv-exponential-lower}
\end{equation}
This bound holds for every \(\rho\ge0\). In the present regime,
it gives \({\TV}(\mathbf P,\mathbf Q)\gtrsim1\), while
\eqref{eq:A-rho} gives \(A\gtrsim1\). Conversely,
\({\TV}(\mathbf P,\mathbf Q)\le1\). Therefore
\[
{\TV}(\mathbf P,\mathbf Q)\asymp\min\{1,A\}.
\]
This completes the proof of Theorem \ref{thm:score}.
\end{proof}
\bigskip
\begin{proof}[Proof of Theorem~\ref{thm:proxy}]
The proof is a straightforward application of Lata\l a's moment theorem~\cite{Latala1997}.

\(G\) is continuous and strictly decreasing. Moreover, \(G(t)\to0\) as \(t\to\infty\) and \(G(t)\to\infty\) as \(t\downarrow0\). (To check the latter fact, bound each summand of $G$ below by $\E|U_i|/t$.) Therefore, by the intermediate value theorem, \(T\) exists and is unique.

Applying Lata\l a's theorem~\cite[Theorem~1]{Latala1997} to the independent nonnegative random variables \(U_i^2\), with \(p=1/2\), gives
\[
\E\Bigl(\sum_{i=1}^nU_i^2\Bigr)^{1/2}\asymp T.
\]
Combined with Theorem~\ref{thm:score}, this gives \eqref{eq:proxy-TV}.
\end{proof}

\section{Examples and applications}
\label{sec:applications}

\subsection{Continuous and discrete distributions}

Theorem \ref{thm:proxy} applies for completely general distributions $P_i$ and $Q_i$. Let's now specialize this result for continuous and for discrete distributions on $\mathbb{R}^d$.

If $P_i$ and $Q_i$ have densities $p_i$ and $q_i$,  then, unwrapping the definition of $U_i$ in \eqref{eq: U}, we get
\[
G(t)=\sum_{i=1}^n\log\left[
\frac12\int
\sqrt{(p_i(x)+q_i(x))^2+
\frac{(p_i(x)-q_i(x))^2}{t^2}}\,\dd x
\right].
\]
Similarly, if $P_i$ and $Q_i$ have probability mass functions $p_i$ and $q_i$, then the same formula holds with the integral replaced by the sum over all $x$ in the joint support.

In either case, ${\TV}(\bigotimes_i P_i, \bigotimes_i Q_i)\asymp \min\{1,T\}$ where $T$ is the unique solution of $G(T)=\frac{1}{2}$.

\subsection{Tensorization bounds}

Suppose we wish to express the TV distance between the product measures $\bigotimes_i P_i$ and $\bigotimes_i Q_i$ in terms of the TV distances between their marginals $P_i$ and $Q_i$. In general, this is not possible: TV distance does not tensorise, even approximately. The best we can do is the following pair of bounds, due to A.~Kontorovich~\cite[Theorem~1.1]{KontorovichTensorization2025}, which we can now readily recover from  Theorem~\ref{thm:score}.

\begin{corollary}[Kontorovich] For any probability measures $P_i$, $Q_i$ on the same measurable space,
\begin{equation}
\min \Big\{1,\Big( \sum_{i=1}^n\TV(P_i,Q_i)^2 \Big)^{1/2}\Big\}
\lesssim
{\TV}\Big(\bigotimes_{i=1}^n P_i,\bigotimes_{i=1}^n Q_i\Big)
\lesssim
\min \Big\{1,\sum_{i=1}^n \TV(P_i,Q_i)\Big\}.
\label{eq:power-envelopes}
\end{equation}
\end{corollary}

\begin{proof}
We just apply Theorem~\ref{thm:score} and then use two elementary bounds on $\E\big(\sum_{i=1}^nU_i^2\big)^{1/2}$. For the upper bound, since $\|x\|_2\le \|x\|_1$ for any vector $x\in \mathbb{R}$, we have
\begin{equation}
\E\Bigl(\sum_{i=1}^nU_i^2\Bigr)^{1/2}\le \E\sum_{i=1}^n|U_i|=\sum_i \TV(P_i,Q_i),
\end{equation}
where we used \eqref{eq: TV as U} in the last step.
For the lower bound, use the convexity of $\ell^2$-norm and Jensen's inequality:
\[
\E\Bigl(\sum_{i=1}^nU_i^2\Bigr)^{1/2}\ge \Bigl(\sum_{i=1}^n(\E |U_i|)^2\Bigr)^{1/2}=\Bigl(\sum_{i}\TV(P_i,Q_i)^2\Bigr)^{1/2}. \qedhere
\]
\end{proof}

\begin{remark}[Optimality]
\label{rem:power-envelopes}
Both bounds in \eqref{eq:power-envelopes} can be attained up to
universal constants even in the i.i.d.\ case; see the
Bernoulli examples in
\cite[Section~2.4]{KontorovichTensorization2025}.
\end{remark}

\subsection{Identical distributions}
Theorem~\ref{thm:proxy} becomes especially transparent if all marginal measures are the same: $$
P_i = P, \quad Q_i = Q \quad \text{for all } i=1, \ldots, n.
$$
Then $U_i\overset{d} = U$,  and we get
\begin{equation}\label{eq: TV powers}
{\TV}(P^{\otimes n},Q^{\otimes n})
\asymp \min\{1,T\},
\end{equation}
where \(T>0\) is the unique solution of
\begin{equation}
\E\sqrt{1+\frac{U^2}{T^2}}=e^{1/(2n)}.
\label{eq:iid-T}
\end{equation}
To simplify this equation,
subtract $1$ from both sides and use that
$\sqrt{1+x^2}-1
\asymp \min\{x^2,x\}$ for $x \ge 0$.
This gives
\begin{equation}
\Psi(T)\asymp \frac1n,
\qquad\text{where}\qquad
\Psi(t):=\E\min\left\{\frac{U^2}{t^2},\frac{|U|}{t}\right\}.
\label{eq:iid-truncated-T}
\end{equation}
The elementary scaling bounds
\begin{equation}
a^{-2}\Psi(t)\le\Psi(at)\le a^{-1}\Psi(t),
\qquad a\ge1,
\label{eq:iid-psi-scaling}
\end{equation}
show that \(\Psi(t)\asymp 1/n\) if and only if \(t\asymp T\).
Thus any approximate solution of \eqref{eq:iid-truncated-T}
may replace \(T\) in \eqref{eq: TV powers}.

\subsection{Hypothesis testing}

Under equal priors, the optimal probability of correctly identifying
whether $n$ independent samples come from $P$ or $Q$ is
\[
\frac12 \Big( 1+\TV(P^{\otimes n},Q^{\otimes n}) \Big),
\]
see \cite{LeCamYang2000}. Thus, the sample complexity for success
probability at least $(1+\delta)/2$ is
\[
n_\delta(P,Q):=\min \bigl\{n\ge1:\TV(P^{\otimes n},Q^{\otimes n})\ge\delta \bigl\}.
\]
For fixed $\delta\in(0,1)$, the classical characterization (see \cite{Suresh2021} and Remark~\ref{rem: higher confidence} below) is in terms of the Hellinger distance:
\begin{equation}	\label{eq: high confidence sample}
	n_\delta(P,Q)\asymp_\delta H(P,Q)^{-2}
\end{equation}
where
\[
H(P,Q)^2=1-\E\sqrt{1-U^2}\asymp\E U^2.
\]
Our refinement is uniform as $\delta\downarrow0$: the
\emph{weak-detection regime}, in which Hellinger distance alone
does not characterize the sample complexity
\cite[Section~8]{PensiaJogLoh2025}.

\begin{corollary}[Uniform weak detection]
\label{cor:weak-detection}
For $P\ne Q$ and $0<\delta\le1/2$,
\begin{equation}
n_\delta(P,Q)
\asymp
\max\left\{
1,\,
\frac{\delta^2}{\E\min\{U^2,\delta|U|\}}
\right\}.
\label{eq:testing-sample-complexity}
\end{equation}
\end{corollary}

\begin{proof}
Let \(\Psi\) be as in \eqref{eq:iid-truncated-T}, and define
\(T_n\) by \(\Psi(T_n)=1/n\). By \eqref{eq: TV powers},
\eqref{eq:iid-truncated-T}, and \eqref{eq:iid-psi-scaling},
\[
c\min\{1,T_n\}\le \TV(P^{\otimes n},Q^{\otimes n})\le C\min\{1,T_n\}
\]
for universal \(0<c\le1/2\) and \(C\ge1\).
For \(0<\delta\le c\), monotonicity of \(\Psi\) gives
\begin{equation}
\label{eq:testing-inversion-bounds}
\max\{1,\Psi(\delta/C)^{-1}\}
\lesssim n_\delta(P,Q)
\lesssim \max\{1,\Psi(\delta/c)^{-1}\}.
\end{equation}
By \eqref{eq:iid-psi-scaling},
both endpoints are comparable to \(\max\{1,\Psi(\delta)^{-1}\}\), proving \eqref{eq:testing-sample-complexity}.

The lower bound in \eqref{eq:testing-inversion-bounds} remains valid for \(c<\delta\le1/2\).
In this range, \(\Psi(\delta)\asymp\E U^2\), and
\eqref{eq:tv-exponential-lower} gives
\[
\TV(P^{\otimes n},Q^{\otimes n})\ge1-\exp(-n\E U^2/2).
\]
Thus any \(n\ge 2\log 2/\E U^2\) satisfies
\(\TV(P^{\otimes n},Q^{\otimes n})\ge1/2\ge\delta\).
Taking \(n\) of this order gives
\[
n_\delta(P,Q)
\lesssim \frac1{\E U^2},
\]
which implies the upper bound in \eqref{eq:testing-sample-complexity}.
\end{proof}

The denominator separates small scores, which contribute
quadratically, from large scores, which contribute linearly:
\[
\E\min\{U^2,\delta|U|\}
=
\E[U^2\mathbf1_{\{|U|\le\delta\}}]
+\delta\,\E[|U|\mathbf1_{\{|U|>\delta\}}].
\]
This distinction is lost when the denominator is replaced by
$\E U^2\asymp H^2(P,Q)$.

\begin{remark}[Higher confidence]	\label{rem: higher confidence}
For completeness, the remaining range has a classical
characterization. Put
$$
a:=-\log\E\sqrt{1-U^2}
= - \log \big( 1 - H^2(P,Q) \big).
$$
Hellinger tensorization and Cauchy--Schwarz give
\[
1-e^{-na}\le \TV(P^{\otimes n},Q^{\otimes n})\le\sqrt{1-e^{-2na}},
\]
and hence, uniformly for $1/2<\delta<1$,
\[
n_\delta(P,Q)
\asymp
\max\left\{1,\frac{\log(1/(1-\delta))}{a}\right\}.
\]
For mutually singular $P,Q$, take $a=\infty$; one sample suffices. In the other end of the spectrum, if $P,Q$ are close (for example, if $H^2(P,Q) \le 1/2$), we have $a \asymp H^2(P,Q)$, so we recover \eqref{eq: high confidence sample}.

Together with \eqref{eq:testing-sample-complexity}, this covers
every success probability strictly between $1/2$ (a random guess) and $1$ (full confidence).
\end{remark}

\subsection{An analytic solution}
To facilitate further applications, let us solve the approximate equation \eqref{eq:iid-truncated-T} analytically.
\begin{lemma}
We have
\begin{equation}
T\asymp
\inf_{A\in\mathcal F}
\left\{
n\E[|U|\mathbf1_A]
+
\sqrt{n\E[U^2\mathbf1_{A^c}]}
\right\}.
\label{eq:iid-split}
\end{equation}
\end{lemma}
\begin{proof}

By \eqref{eq:iid-truncated-T}, for every \(A\in\mathcal F\), we have
\[
\frac{1}{n}\asymp \E\min\left\{\frac{U^2}{T^2},\frac{|U|}{T}\right\} \lesssim
\frac 1T\E[|U|\mathbf1_A]+\frac 1 {T^2}\E[U^2\mathbf1_{A^c}].
\]
Thus
\[
T\lesssim \max\left\{ n\E[|U|\mathbf1_A], \sqrt{n\E[U^2\mathbf1_{A^c}]}\right\} \lesssim n\E[|U|\mathbf1_A] +\sqrt{n\E[U^2\mathbf1_{A^c}]}.
\]
Meanwhile, taking \(A=\{|U|>T\}\), by \eqref{eq:iid-truncated-T}, we obtain
\[
\frac{1}{n}\asymp \E\min\left\{\frac{U^2}{T^2},\frac{|U|}{T}\right\}= \frac 1T\E[|U|\mathbf1_A] + \frac 1{T^2}\E[U^2\mathbf1_{A^c}].
\]
Hence,
\[T\gtrsim \max\left\{ n\E[|U|\mathbf1_A], \sqrt{n\E[U^2\mathbf1_{A^c}]}\right\} \gtrsim n\E[|U|\mathbf1_A] +\sqrt{n\E[U^2\mathbf1_{A^c}]}.\]
Thus we finish the proof.
\end{proof}

\subsection{Multinomial laws}

For a probability measure \(P\) on \([k]:=\{1,\ldots,k\}\), let \(\operatorname{Mult}(n,P)\) denote
the law of the histogram of \(n\) independent samples from \(P\).

For two such measures \(P,Q\), every word with histogram \(z\) has probability \(\prod_jP(j)^{z_j}\) under \(P^{\otimes n}\) and \(\prod_jQ(j)^{z_j}\) under \(Q^{\otimes n}\).
Their difference has the same sign for all words with the same
histogram. Thus taking histograms preserves TV, and
Theorem~\ref{thm:proxy} gives
\begin{equation}
\TV\bigl(\operatorname{Mult}(n,P),\operatorname{Mult}(n,Q)\bigr)
=
\TV(P^{\otimes n},Q^{\otimes n})
\asymp\min\{1,T\}.
\label{eq:app-multinomial-sufficiency}
\end{equation}
By \eqref{eq:iid-split}, we have
\begin{equation}
T\asymp
\min_{A\subseteq[k]}
\left\{
n\sum_{j\in A}|P(j)-Q(j)|+
\sqrt{n\sum_{j\in A^c}\frac{(P(j)-Q(j))^2}{P(j)+Q(j)}}
\right\}.
\label{eq:app-multinomial-split}
\end{equation}

For Bernoulli case, let \(p\ne q\). In \eqref{eq:app-multinomial-split}, both
absolute differences equal \(|p-q|\), and the two denominators are \(p+q\) and \(2-p-q\). Together with
\eqref{eq:app-multinomial-sufficiency}, this yields\begin{equation}
\TV\bigl(\Ber(p)^{\otimes n},\Ber(q)^{\otimes n}\bigr)
\asymp
\min\left\{
1,\,
n|p-q|,\,
\frac{\sqrt n\,|p-q|}
{\sqrt{(p+q)(2-p-q)}}
\right\}.
\label{eq:Bernoulli-scale}
\end{equation}
The case \(p=q\) is immediate. This recovers the estimate of
Liu et al.~\cite[Theorem~5]{LiuEtAl2020}; see
\cite{AdellJodra2006} for exact formulas.

\subsection{Heterogeneous categorical sums}

Let \(P_i,Q_i\) be probability measures on \([k]\), with masses
\(p_{ij}:=P_i(\{j\})\) and \(q_{ij}:=Q_i(\{j\})\).
For independent \(X_i\sim P_i\) and \(Y_i\sim Q_i\), set
\[
N_P:=\sum_{i=1}^n e_{X_i},
\qquad
N_Q:=\sum_{i=1}^n e_{Y_i},
\]
where \(e_j\) is the \(j\)th standard basis vector of \(\mathbb R^k\).

\subsubsection*{Common tilt.}
Suppose that, for some \(h\in\mathbb R^k\),
\begin{equation}
p_{ij}=\frac{q_{ij}e^{h_j}}{Z_i},
\qquad
Z_i:=\sum_{\ell=1}^k q_{i\ell}e^{h_\ell}.
\label{eq:app-common-tilt}
\end{equation}
The two product laws have the same support. For a word \(x\)
in this support with histogram \(z\),
\[
\prod_{i=1}^n\frac{p_{i,x_i}}{q_{i,x_i}}
=
\frac{\exp\{\sum_jh_jz_j\}}{\prod_iZ_i}.
\]
Hence their probability differences have the same sign for
all words with the same histogram. Summing within each
histogram therefore preserves total variation.
Theorem~\ref{thm:proxy} gives
\begin{equation}
{\TV}(N_P,N_Q)
=
{\TV}\big(\bigotimes_i P_i,\bigotimes_i Q_i\big)
\asymp\min\{1,T\},
\label{eq:app-common-tilt1}
\end{equation}
\begin{remark}
\label{rem:app-tilt-characterization}
If all \(p_{ij},q_{ij}\) are positive, then
\eqref{eq:app-common-tilt} holds if and only if the product probability ratio depends only on the histogram. To see the converse, swapping
categories \(j,\ell\) between distinct trials \(i,r\) gives
\[
\frac{p_{ij}}{q_{ij}}\frac{p_{r\ell}}{q_{r\ell}}
=
\frac{p_{i\ell}}{q_{i\ell}}\frac{p_{rj}}{q_{rj}}.
\]
Thus the ratio matrix has rank one:
\(p_{ij}/q_{ij}=a_i b_j\) with \(a_i,b_j>0\).
Row normalization gives \(a_i=(\sum_j q_{ij}b_j)^{-1}\), yielding
\eqref{eq:app-common-tilt} with \(h_j=\log b_j\).
\end{remark}

\subsection*{Use of generative AI}
Generative-AI tools were used during exploratory development and editorial preparation of this draft. The authors take responsibility for all mathematical statements and references.

\subsection*{Overlap with an earlier preprint}
Some results overlap with \cite{KontorovichHomogenizationPrinciple2026}. The latter preprint is not currently under submission and will not be submitted for publication at any future time.

{\footnotesize
\newcommand{\etalchar}[1]{$^{#1}$}

}

\end{document}